\documentclass[reqno,11pt]{amsart}
\usepackage{amsmath,amsthm,mathtools}
\usepackage{esint}
\usepackage{mathabx}
\usepackage{subfigure}
\allowdisplaybreaks[4]
\newcommand{\ud}{\mathrm{d}}

\newtheorem{theorem}{Theorem}[section]

\newtheorem{lemma}{Lemma}[section]
\newtheorem{remark}{Remark}[section]
\newtheorem{definition}{Definition}[section]
\newtheorem{proposition}{Proposition}[section]

\begin{document}
\title[Stability of critical points for Onsager functional]{On stability of critical points for the high dimensional Onsager functional with Maier-Saupe potential}

\author{Sisi Guan}
\author{Wei Wang}
\author{Qi Zhao}
\address{School of Mathematical Sciences, Zhejiang University, Hangzhou 310058, China}
\email{sisiguan@zju.edu.cn}
\email{wangw07@zju.edu.cn}
\email{zhaoq27@zju.edu.cn}

\renewcommand{\theequation}{\thesection.\arabic{equation}}
\setcounter{equation}{0}

%\allowdisplaybreaks[4]
%%%%%%%%%%%%%%%%%%%%%%%%%%%%%%%%%%%%%%%%%%%%%%

\begin{abstract}
We study the stability of the critical points of the Onsager energy functional with Maier-Saupe interaction potential in general dimensions. We show that the stable critical points must be axisymmetric, which solves a problem proposed by [H. Wang, P. Hoffman, Commun. Math. Sci. 6 (2008), 949--974]  and further conjectured by [P. Degond, A. Frouvelle, J.-G. Liu, Kinetic and Related Models, 15(2022), 417--465]. The main ingredients of the proof include a suitable decomposition of the second variation around the critical points and a detailed analysis of the relation between the intensity of interaction and the order parameter characterizing the anisotropy of the solution.
\end{abstract}
%\date{\today}
\maketitle

\numberwithin{equation}{section}

\section{Introduction}
We study the stability of all critical points to the energy functional on the sphere $\mathbb{S}^{n-1}$:
\begin{align}\label{energy:onsager}
  \mathcal{A}[f]=\int_{\mathbb{S}^{n-1}}f(m)\log f(m)\ud m+\frac{\alpha}{2}\int_{\mathbb{S}^{n-1}\times\mathbb{S}^{n-1}}
  \big(1-(m\cdot m')^2\big) f(m) f(m')\ud m \ud m'.
\end{align}
Here $f(m)$ is a positive function on $\mathbb{S}^{n-1}$  with $\int_{\mathbb{S}^{n-1}}f(m)\ud m=1$ and $\alpha$ is a positive parameter.
The Euler-Lagrange equation reads as
\begin{align}\label{eq:EL}
  \ln f(m)-\alpha \int_{\mathbb{S}^{n-1}}(m\cdot m')^2 f(m')\ud m'=\mathrm{const.}\quad \text{ on }\mathbb{S}^{n-1}.
\end{align}
A solution to (\ref{eq:EL}) is called a {\it critical point} of $\mathcal{A}[f]$.
Let $I_n$ be the identity matrix and
\begin{align}\label{def:M}
  M=\alpha\int_{\mathbb{S}^{n-1}}\Big(m\otimes m-\frac{1}{n}I_n\Big)f(m)\ud m,
\end{align}
which is a trace free $n\times n$ symmetric matrix.
Then by \eqref{eq:EL} we can write
\begin{align} \label{eq:f}
  f(m)=\frac{1}{Z}e^{M:(m\otimes m)},\quad \text{with } Z= \int_{\mathbb{S}^{n-1}}e^{M:(m\otimes m)}\ud m.
\end{align}
Combining \eqref{def:M}-\eqref{eq:f}, we get that $M$ satisfies the nonlinear equation:
\begin{align}\label{eq:M}
 M =\frac{\alpha\int_{\mathbb{S}^{n-1}}\Big(m\otimes m-\frac{1}{n}I_n\Big)e^{M:(m\otimes m)}\ud m}{
  \int_{\mathbb{S}^{n-1}}e^{M:(m\otimes m)}\ud m}.
\end{align}

When $n=3$, the functional \eqref{energy:onsager} is the well-known Onsager energy functional with the Maier-Saupe interaction potential $1-(m\cdot m')^2$.
It is a fundamental model and has been widely used to characterize the phase transition from the isotropic state to the nematic state for liquid crystals.
The function $f(m)$ represents the probability distribution of rod-like molecules with orientation $m\in\mathbb{S}^{2}$ for a nematic liquid crystal. The first term in \eqref{energy:onsager} is due to entropy, while the second characterizes interaction energy between pairs of molecules and the parameter $\alpha$ measures the molecular interaction intensity.
The critical points of \eqref{energy:onsager} for $n=3$ have been completely classified in \cite{LZZ, FS}. In particular, it is shown that
all critical points are axially symmetric, i.e., they must be functions depending only on $m\cdot \nu$ for some $\nu\in\mathbb{S}^{2}$. Some different proofs are given in \cite{ZWFW, Bal}. For the stability of these critical points (for $n=3$), we refer to \cite{ZW} for a direct proof and \cite{WZZ} for the estimate of the second variation near the critical points.

The high dimensional ($n\ge4$) Onsager model was introduced by Wang-Hoffman \cite{WH}, in which it is proved: if a symmetric trace free matrix $M$ satisfies \eqref{eq:M}, then $M$ has only two distinct eigenvalues $\lambda_1$ and $\lambda_2$. For the completeness of this paper, we present another proof in Appendix by the method of J. Ball \cite{Bal} used for $n=3$. Assume that $\lambda_1$ and $\lambda_2$ occur $k$ and $n-k$ times ($1\le k\le n-1$) respectively. Then
$k\lambda_1+(n-k)\lambda_2=0$, and we have $\lambda_1=\frac{\eta(n-k)}{n}$, $\lambda_2=\frac{-\eta k}{n}$ for some $\eta\in\mathbb{R}$.

Firstly,  we assume that $M$ is diagonal, i.e.,
\begin{align}\nonumber
  M=\frac{\eta}{n}\mathrm{diag}\, \big\{\underbrace{n-k,\cdots, n-k}_{k \text{ times}},\underbrace{-k,\cdots,-k}_{n-k \text{ times}}\big\}.
\end{align}
For $m\in\mathbb{S}^{n-1}$, we use the representation
\begin{align}\label{decomp:m}
m=(\sin\theta\omega, \cos\theta\xi)\text{ with }\omega=(\omega_i)_{1\le i\le k}\in\mathbb{S}^{k-1},\  \xi=(\xi_j)_{1\le i\le n-k}\in\mathbb{S}^{n-k-1},\ \theta\in[0,\frac\pi2].
\end{align}
The area element of $\mathbb{S}^{n-1}$ can be written as
 \begin{align}\label{decomp:measure}
   \ud m=\ud\mu_k^{\theta}\ud\omega\ud \xi,
   \quad \text{where }\ud\mu_k^{\theta}:=\sin^{k-1}\theta \cos^{n-k-1}\theta \ud\theta.
 \end{align}
Then we obtain from \eqref{eq:f} that the corresponding distribution function takes the form:
\begin{equation}\label{equ:critical-0}
f(m)=h^{(k)}_{\eta}(m):=\frac{e^{\lambda_1\sum_{i=1}^k m_i^2+\lambda_2\sum_{i=k+1}^n m_i^2}}{\int_{\mathbb{S}^{n-1}}e^{\lambda_1\sum_{i=1}^k m_i^2+\lambda_2\sum_{i=k+1}^n m_i^2}\ud m}%=\frac{e^{\eta\sum_{i=1}^k m_i^2}}{\int_{\mathbb{S}^{n-1}}e^{\eta\sum_{i=1}^k m_i^2} \ud m}
=\frac{e^{\eta\sin^2\theta}}{\int_{\mathbb{S}^{n-1}}e^{\eta\sin^2\theta} \ud m}.
\end{equation}
 Equation (\ref{eq:M}) gives the relation between $\eta$ and $\alpha$ for given $k=1,2,\cdots, n-1$:
\begin{align}\label{eqn:eta-alpha}
\frac{(n-k)\eta}{n}\int_0^{\frac{\pi}{2}}e^{\eta\sin^2\theta} \ud\mu_k^{\theta}
=\frac{\alpha}{nk}\int_0^{\frac{\pi}{2}}e^{\eta\sin^2\theta}(n\sin^2\theta-k)\ud\mu_k^{\theta}.
\end{align}
Equation (\ref{eqn:eta-alpha}) holds if and only if $\eta=0$ or
\begin{align}\nonumber
\alpha=\sigma_k(\eta):=\frac{(n-k)k\int_0^{\frac{\pi}{2}}e^{\eta\sin^2\theta}\ud\mu_k^{\theta}}{2\int_0^{\frac{\pi}{2}}e^{\eta\sin^2\theta}\sin^2\theta\cos^2\theta \ud\mu_k^{\theta}}.
\end{align}
Here, we have used the relation:
\begin{align*}
2\eta  \int_0^{\frac{\pi}{2}}e^{\eta\sin^2\theta}\sin^2\theta\cos^2\theta \ud\mu_k^{\theta}
-\int_0^{\frac{\pi}{2}}e^{\eta\sin^2\theta}(n\sin^2\theta-k)\ud\mu_k^{\theta}\\
= \int_0^{\frac{\pi}{2}}\partial_\theta(e^{\eta\sin^2\theta}\sin^{k}\theta \cos^{n-k}\theta )\ud \theta=0.
\end{align*}
Apparently, there holds $\sigma_k(\eta)=\sigma_{n-k}(-\eta)$.

Note that, when $\eta=0$, the corresponding distribution $f(m)$ is just $\frac{1}{|\mathbb{S}^{n-1}|}$, which implies that the uniform distribution $\frac{1}{|\mathbb{S}^{n-1}|}$ is always a critical point for all $\alpha>0$.

For general $M$ which satisfies \eqref{eq:M}, there exists  $R\in O(n)$ such that
\begin{align*}
M=R\,\mathrm{diag}\,\Big\{\underbrace{\frac{\eta(n-k)}{n},\cdots, \frac{\eta(n-k)}{n}}_{k \text{ times}},\underbrace{-\frac{\eta k}{n},\cdots,-\frac{\eta k}{n}}_{n-k \text{ times}}\Big\}\,R^T.
\end{align*}
Thus, the corresponding equilibrium distribution $f$ in (\ref{eq:f}) can be written as
\begin{align}\label{equ:critical-1}
 f(m)=h^{(k)}_{\eta,R}(m):= \frac{e^{\eta|P_k Rm|^2}}{ \int_{\mathbb{S}^{n-1}}e^{\eta|P_k Rm|^2}\ud m}.
\end{align}
Here $P_k$ is the projection operator: $P_km=(m_1,m_2,\cdots, m_k,0,\cdots, 0)$ for $m=(m_1,m_2,\cdots, m_n)$.
In particular, for $k=1$, let $\nu=(R_{1i})_{1\le i\le n}\in \mathbb{S}^{n-1}$. Then we may write $ h^{(1)}_{\eta,R}(m)$ as
\begin{align*}
\frac{e^{\eta(m\cdot\nu)^2}}{ \int_{\mathbb{S}^{n-1}}e^{\eta(m\cdot\nu)^2}\ud m}\triangleq h_\nu(m),
\end{align*}
which is axially symmetric (with axis $\nu\in \mathbb{S}^{n-1}$).

By the above discussion, we can obtain the complete classification of critical points to \eqref{energy:onsager}.
\begin{proposition}[\cite{WH}]
All the critical points of $\mathcal{A}[f]$ are given by the uniform distribution $h_0:=\frac{1}{|\mathbb{S}^{n-1}|}$ and $h^{(k)}_{\eta,R}(m)$ (defined in \eqref{equ:critical-1}), where $R\in O(n)$,  $k\in\{1,2,\cdots, [\frac{n}{2}]\}, \eta\in\mathbb{R}$ which satisfy
\begin{align*}
  \alpha =\sigma_k(\eta).
\end{align*}
\end{proposition}
The critical point $h_0$ is called an {\it isotropic} equilibria and $h^{(k)}_{\eta,R}(m)$ is called {\it anisotropic}. The parameter $\eta$ is called order parameter which describes the anisotropy of critical points $h^{(k)}_{\eta,R}(m)$.

For general $n\ge 4$, it is unclear that whether these solutions are stable except the case $n =4$ solved by Frouvelle \cite{Frou}. In particular, Degond-Frouvelle-Liu \cite{DFL} conjectured that the only stable anisotropic equilibria occurs when $k=1$ and $\eta>\eta_1^*:=\mathrm{argmin }\,\sigma_1$, and thus must be axisymmetric. They formally derived the Ericksen-Leslie model from the dynamical Doi-Onsager model for nematic liquid crystal flow in general dimensions by assuming the conjecture is true.

In this paper, we give a rigorous proof to this conjecture.
Our first main result can be stated as follows, which shows that, the stable anisotropic critical points must be axisymmetric.
Define the space of perturbations:
$$V=\Big\{\phi(m)\in L^2(\mathbb{S}^{n-1}):\ \int_{\mathbb{S}^{n-1}}\phi(m) \ud m=0\Big\}.$$
\begin{definition}
A critical point $f_0$ of $ \mathcal{A}[f]$ is called to be stable, if the corresponding second variation around it is nonnegative definite, i.e.,
\begin{align*}
\Big\langle \frac{\delta^2 \mathcal{A}}{\delta f^2}\Big|_{f=f_0} \phi,\  \phi\Big\rangle_{L^2(\mathbb{S}^{n-1})} \ge 0,\quad \text{ for any }\phi\in V.
\end{align*}
\end{definition}
\begin{theorem}\label{thm:1}
(i) The isotropic equilibiria $h_0=\frac{1}{|\mathbb{S}^{n-1}|}$ is stable if and only if $\alpha\le \frac{n(n+2)}{2}$;\\
(ii) All $h^{(k)}_{\eta,R}(m)$ for $2\le k\le [\frac{n}{2}]$ are not stable;\\
(iii) For $k=1$, $h^{(1)}_{\eta,R}(m)$ is stable when $\eta>\eta_1^*$ and unstable when $\eta<\eta_1^*$, where $\eta_1^*:=\mathrm{argmin }\,\sigma_1$.
\end{theorem}
\begin{remark}
Let $\tilde{R}\in O(n)$ be a reflection defined by $R:(m_1, m_2,\cdots, m_n)\mapsto (m_n, m_{n-1},\cdots, m_1)$. Then one has
that $h^{(n-k)}_{\eta,R}(m)=h^{(k)}_{-\eta,\tilde{R}R}(m)$. Therefore, for $k>[\frac{n}{2}]$, the only stable anisotropic critical points are $h^{(n-1)}_{\eta,R}(m)$ with $\sigma_{n-1}'(\eta)<0$, which indeed are equivalent to the equilibria $h^{(1)}_{-\eta,\tilde{R}R}(m)$.
\end{remark}

For a given $k$, there is a unique $\eta^*_k$ such that  $\sigma'_k(\eta^*_k)=0$ and  $\sigma'_k(\eta)\lessgtr0$ if and only if    $\eta\lessgtr \eta^*_k$  (see Lemmas \ref{lem:single-root}-\ref{lem3}). Moreover, when $\alpha<\sigma_k(\eta^*_k)=\min_{\eta\in\mathbb{R}} \sigma_k(\eta)$, the relation $\alpha=\sigma_k(\eta)$
can not be satisfied for any $\eta\in\mathbb{R}$, thus the critical points $h_{\eta,R}^{(k)}$ do not exist. In particular, if $\alpha<\sigma_1(\eta^*_1)$, then $h_0$ is the only critical point. We refer to Figure \ref{fig:phase1} for the graphs of the relations $\alpha=\sigma_k(\eta)$ for different $k$.

\begin{figure}
\centering \mbox{
\subfigure[$n=5$]{\includegraphics[height=7.5cm]{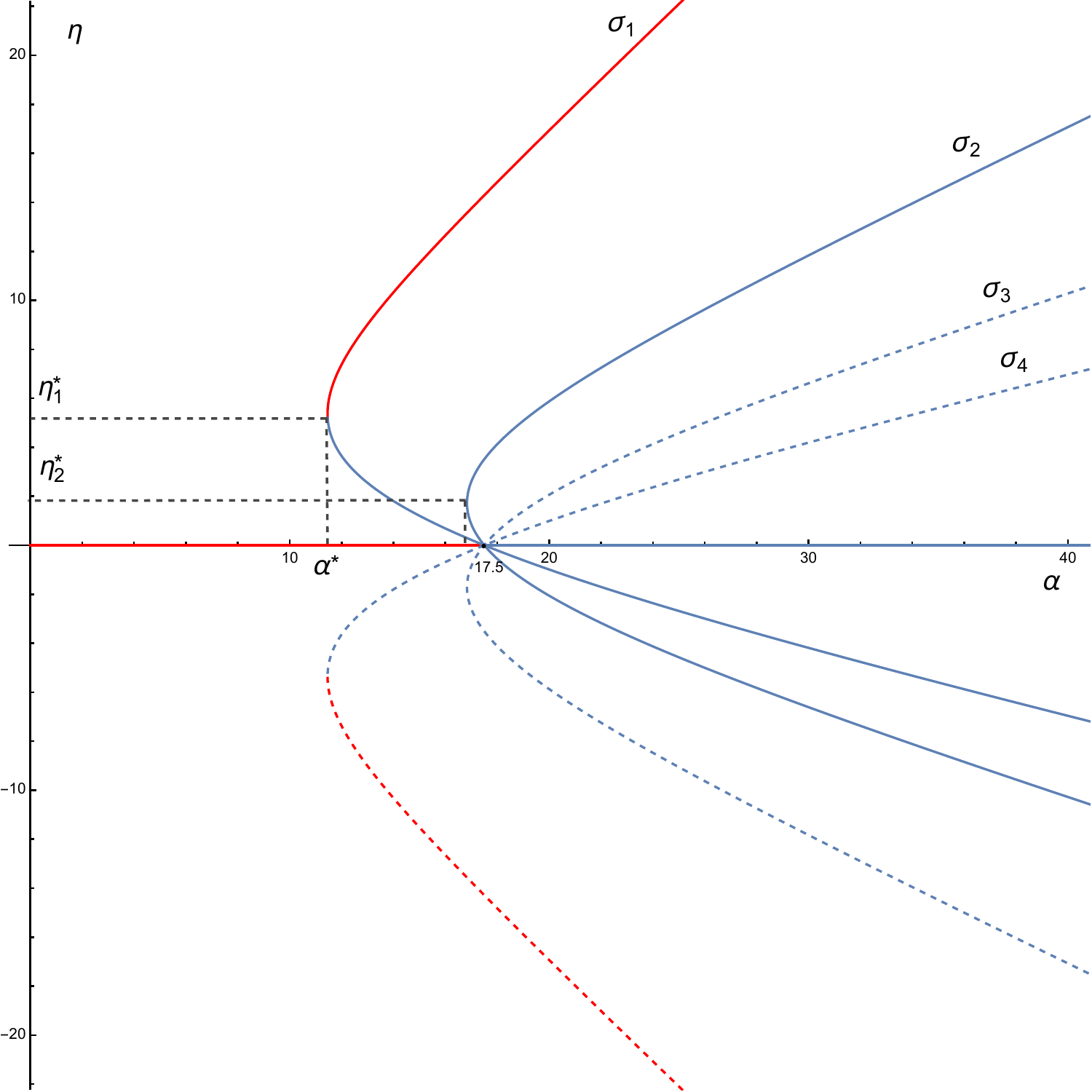}}
\quad\quad
\subfigure[$n=6$]{\includegraphics[height=7.5cm]{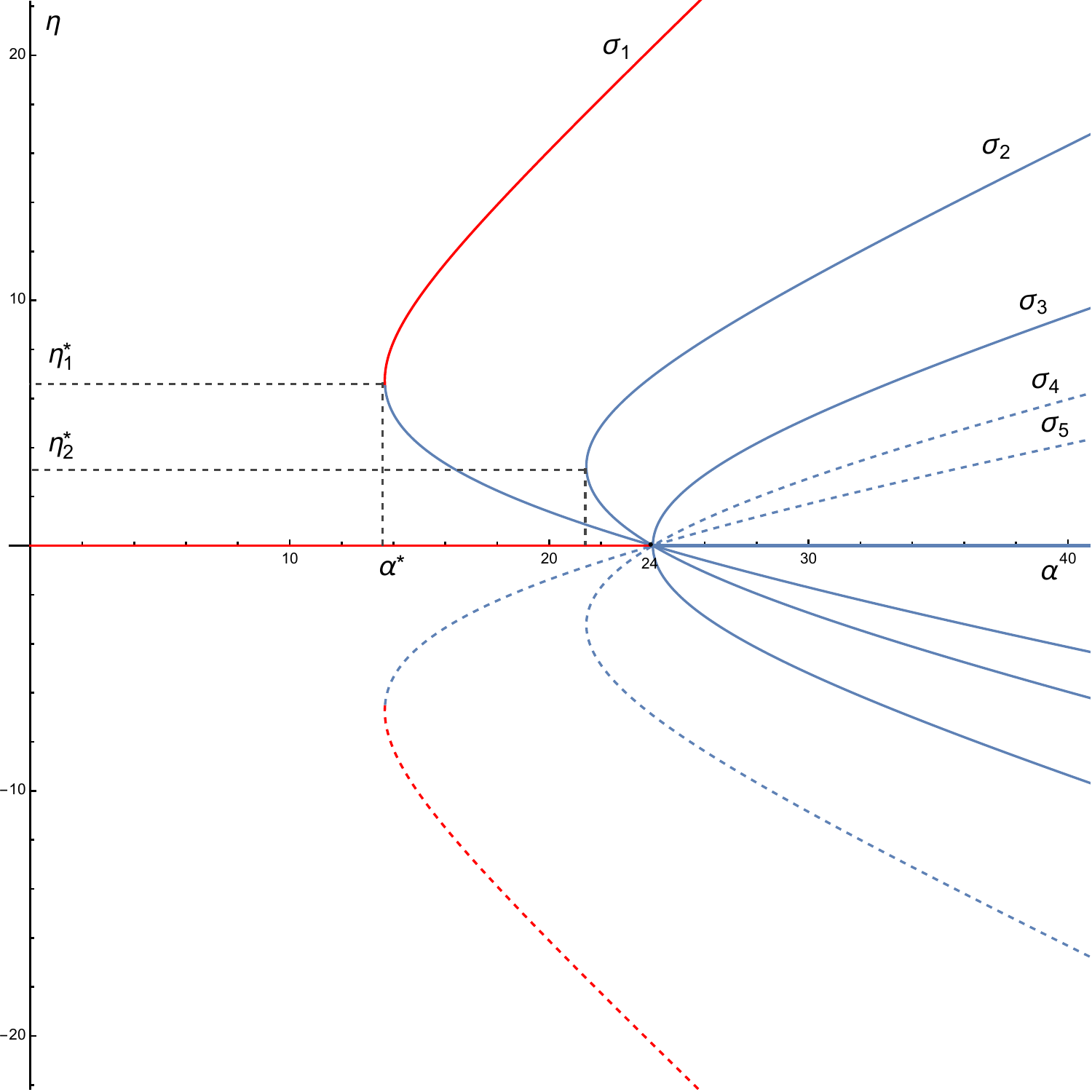}}}
\caption{The graphic characterization for $\eta$ and $\alpha$ satisfying \eqref{eqn:eta-alpha} for $n=5$ and $n=6$. The red curves corresponds to stable solutions, and blue ones correspond to unstable solutions. The dashed curves do not represent new solutions due to $h^{(n-k)}_{\eta,R}=h^{(k)}_{-\eta,\tilde{R}R}$.}
\label{fig:phase1}
\end{figure}

Giving a critical point $f_0$, let
\begin{align*}
  \mathcal{H}_{f_0}\phi=\frac{\delta^2 \mathcal{A}}{\delta f^2}\Big|_{f=f_0} \phi.
\end{align*}
It is direct to calculate that for $\phi\in V$:
\begin{align*}
\mathcal{H}_{f_0}\phi
=&\frac{\phi}{f_0}-\alpha\int_{\mathbb{S}^{n-1}}
  (m\cdot m')^2 \phi(m') \ud m'
  -\fint_{\mathbb{S}^{n-1}}\Big( \frac{\phi}{f_0}-\alpha\int_{\mathbb{S}^{n-1}}(m\cdot m')^2 \phi(m') \ud m'\Big)\ud m,
\end{align*}
where $\fint_{\mathbb{S}^{n-1}}=\frac{1}{|\mathbb{S}^{n-1}|}\int_{\mathbb{S}^{n-1}}$ represents the average of the integral on $\mathbb{S}^{n-1}$.
Thus
\begin{align}\nonumber
\langle \mathcal{H}_{f_0}\phi,\phi\rangle_{L^2(\mathbb{S}^{n-1})}
=&\int_{\mathbb{S}^{n-1}}\frac{\phi^2}{f_0}\ud m-\alpha\int_{\mathbb{S}^{n-1}\times\mathbb{S}^{n-1}}
  (m\cdot m')^2 \phi(m) \phi(m')\ud m \ud m'\\ \label{def:H}
=&\int_{\mathbb{S}^{n-1}}\frac{\phi^2}{f_0}\ud m-\alpha\Big|\int_{\mathbb{S}^{n-1}}\Big(m\otimes m-\frac1nI_n\Big) \phi(m)\ud m\Big|^2.
\end{align}
The stability of $f_0$ is equivalent to the nonnegativity of the above quadratic form for any perturbation $\phi\in V$.

Due to the rotational symmetry of the energy functional, we only need to study the stability of $h^{(k)}_{\eta}(m)$ (see \eqref{equ:critical-0}) where $R$ is the identity matrix. Thus, in the sequel, we shall assume that $f_0=h^{(k)}_{\eta}(m)$ which depends only on $\theta$ (or equivalently, $|\sum_{i=1}^k m_i^2|$).

Next, for given $\eta>\eta_1^*$ and the stable anisotropic critical point $h_\nu$, we study the kernel space of $\mathcal{H}_{h_\nu}$, which is defined by:
\begin{align*}
  \mathrm{Ker}\ \mathcal{H}_{h_\nu}:=\{\phi\in V: \mathcal{H}_{h_\nu}\phi=0\}.
\end{align*}
Let
\begin{align*}
  V_{\nu}^\top:=&\,\Big\{u \cdot \big[m-\nu(m\cdot \nu)\big](m\cdot\nu) h_\nu(m): u\in \mathbb{R}^n \Big\},\\
  V_{\nu}^\bot:=&\,\big\{\psi \in V: \langle\psi, \phi \rangle=0, \text{ for }\forall \phi\in V_{\nu}^\top\big\}.
\end{align*}
The next theorem gives a characterization on $\mathrm{Ker}\ \mathcal{H}_{h_\nu}$, which plays an important role to study the Ericksen-Leslie limit of the dynamical Doi-Onsager model for nematic liquid crystal flow; see \cite{WZZ, DFL}.
\begin{theorem}\label{thm:2}
If $\eta>\eta_1^*$, then $\mathrm{Ker}\ \mathcal{H}_{h_\nu}=  V_{\nu}^\top$.
Moreover, there exist $c_0(\eta, n)>0$ such that for $\psi\in V_{\nu}^\bot$:
\begin{align}\nonumber
  \langle\mathcal{H}_{h_\nu}\psi, \psi\rangle_{L^2(\mathbb{S}^{n-1})} \ge c_0(\eta,n) \|\psi\|^2_{L^2(\mathbb{S}^{n-1})}.
\end{align}
\end{theorem}

The paper will be organized as follows. In Section \ref{sec:decomp}, we derive a key decomposition formula for the quadratic form \eqref{def:H}; see Proposition \ref{prop:decomp}.
Section \ref{sec:nonnegativity} is devoted to study the nonnegativity of each term arising from the decomposition of \eqref{def:H}. A main step is the analysis on the behavior of $\sigma_k(\eta)$.
Then the proof of Theorems 1.1 and 1.2 will be completed in Section \ref{sec:proof}.

{\bf Notations}: We use $S_d$ to denote the surface area of $\mathbb{S}^{d-1}$. We will denote by $\langle\cdot,\,\cdot \rangle_{\mathcal{M}}$
the $L^2(\mathcal{M})$-inner production for a smooth manifold $\mathcal{M}$. Four typical cases will be frequently used: $\mathcal{M}=\mathbb{S}^{n-1}$, $\mathbb{S}^{k-1}$, $\mathbb{S}^{n-k-1}$, and $\mathbb{S}^{k-1}\times \mathbb{S}^{n-k-1}$ with standard metrics on them.
Furthermore, $\langle\cdot,\,\cdot\rangle_{\mathbb{S}^{n-1}}$ will be simplified as $\langle\cdot,\,\cdot \rangle$.

\section{Decomposition of the second variation}\label{sec:decomp}

For $m\in\mathbb{S}^{n-1}$, we rewrite it as $m=(\sin\theta\omega, \cos\theta\xi)$ with $\omega=(\omega_i)_{1\le i\le k}\in\mathbb{S}^{k-1}$, $ \xi=(\xi_j)_{1\le i\le n-k}\in\mathbb{S}^{n-k-1}$ and $\theta\in[0,\frac\pi2]$.  Note again that the critical points $f_0$ we considered is just a function of $\theta$.

Let $U\subset V$ be the set of all following functions defined on $\mathbb{S}^{n-1}$ (which is independent of the variable $\theta$):
\begin{align}\nonumber
\left\{
\begin{array}{ll}
\Omega_{ii'}(\omega, \xi)=\omega_i\omega_{i'},\quad &\text{ for }1\le i < i'\le k;\\
 \Omega_{0i}(\omega, \xi)=\frac{i\omega_{i+1}^2-\sum_{i'=1}^{i} \omega_{i'}^2}{\sqrt{2i(i+1)}},\quad &\text{ for }1\le i\le k-1;\\
\Xi_{jj'}(\omega, \xi)=\xi_j\xi_{j'},\quad &\text{ for }1\le j < j'\le n-k;\\
 \Xi_{0j}(\omega, \xi)=\frac{j\xi_{j+1}^2-\sum_{j'=1}^{j} \xi_{j'}^2}{\sqrt{2j(j+1)}},\quad &\text{ for }1\le j\le n-k-1;\\
\Theta_{ij}(\omega, \xi)=\omega_i\xi_j,\quad &\text{ for }1\le i\le k, 1\le j\le n-k.
\end{array}\right.
\end{align}
For the simplicity of presentation, we introduce the sets of subscripts for $\Omega$, $\Xi$ and $\Theta$:
\begin{align*}
\Lambda_1=&\,\{0i: 1\le i\le k-1\}\cup \{ii': 1\le i< i'\le k\}\triangleq \Lambda_1^A \cup \Lambda_1^B,\\
\Lambda_2=&\,\{0j: 1\le i\le n-k-1\}\cup \{jj': 1\le j<j'\le n-k\}\triangleq \Lambda_2^A \cup \Lambda_2^B,\\
\Lambda_0= &\,\{ij: 1\le i\le k, 1\le j\le n-k\},
\end{align*}
and the Kronecker symbol $\delta_I^J$ for $I$, $J$ in the same set $\Lambda_i$ is defined by: $\delta_I^J=1$ if $I=J$; $\delta_I^J=0$ if $I\neq J$.
Then we may rewrite the set $U$ as
\begin{align*}
  U=\{\Omega_I(\omega):\ {I\in\Lambda_1} \}\cup \{\Xi_J(\xi):\ {J\in\Lambda_2} \} \cup\{\Theta_K(\omega, \xi):\ {K\in\Lambda_0}\}.
\end{align*}
Note that all these functions depend only on $(\omega, \xi)$. For any $p(\omega,\xi)\in U$, we have
\begin{align*}
\int_{\mathbb{S}^{k-1}\times\mathbb{S}^{n-k-1}}p(\omega,\xi)\ud \omega\ud \xi=0.
\end{align*}
\begin{remark}
  When $k=1$, we have $m=(\sin\theta, \cos\theta\xi)$ with $\xi\in\mathbb{S}^{n-2}$. Thus $\Lambda_1=\emptyset$ and $U$ contains only $\Xi_{(\cdot)}$ and $\Theta_{1j}(=\xi_j)$. Similarly, $\Lambda_2=\emptyset$ if $k=n-1$.
\end{remark}

 We introduce two subspaces of $V$:
\begin{align}\nonumber
{V}^{\top}=&~\mathrm{span}\Big\{a(\theta) p(\omega, \xi),\,b(\theta):\   p(\omega, \xi)\in U;\, a,\, b\in L^2\Big(\big[0,\frac\pi2\big],\ud\mu_k^{\theta}\Big),\, \int_{0}^{\pi/2} b(\theta) \ud\mu_k^{\theta} =0\Big\},
\\
\nonumber {V}^{\bot}=&~\Big\{\psi(\theta, \omega, \xi):\  \psi\in V, \, \langle \psi,\phi\rangle_{L^2} =0, \forall \phi\in V^\top\Big\}.
\end{align}
An important fact is that ${V}^{\top}$ contains all components of $m\otimes m-\frac{1}{n}I_n$.
\begin{lemma}\label{lem:ortho}
All the components of $(m\otimes m-\frac1nI_n)$ belongs to ${V}^{\top}$. Thus, for all $\psi(m)\in {V}^{\bot}$,
\begin{align*}
  \int_{\mathbb{S}^{n-1}}\Big(m\otimes m-\frac{1}{n}I_n\Big)\psi(m)\ud m=0.
\end{align*}
\end{lemma}

\begin{proof}For the off-diagonal components $m_im_{j}(i\neq j)$, we get by direct calculations that
  \begin{align*}
m_im_{i'}&=\sin^2\theta \omega_i\omega_{i'} \in {V}^{\top},\quad \text{ for }1\le i < i'\le k;\\
m_jm_{j'}&=\cos^2\theta \xi_i\xi_{i'} \in {V}^{\top},\quad\text{ for }1\le j < j'\le n-k;\\
m_im_j&=\sin\theta\cos\theta \omega_i\xi_j\in {V}^{\top},\quad \text{ for }1\le i\le k, 1\le j\le n-k.
\end{align*}
For the diagonal components $m_k^2-\frac{1}{k}$, one has
\begin{align*}
\omega_k^2-\frac{1}{k}&= \omega_k^2-\frac{1}{k}\sum_{i=1}^{k}\omega_i^2 = \frac{\sqrt{2k(k-1)}}{k}\Omega_{0,k-1}\in {V}^{\top},\\
\omega_{i+1}^2-\omega_{i}^2&=\frac{\sqrt{2i(i+1)}}{i}\Omega_{0,i}-\frac{\sqrt{2i(i-1)}}{i}\Omega_{0,i-1}\in {V}^{\top},\quad \text{for }i=1,2,\cdots, k-1,
\end{align*}
which implies $\omega_{i}^2-\frac1k \in {V}^{\top} $ for $i=1,2,\cdots, k$.
Thus by the fact that $\sin^2\theta-\frac{k}{n}\in{V}^{\top} $, one yields
\begin{align*}
  m_i^2-\frac{1}{n}=  \sin^2\theta\Big(\omega_i^2-\frac{1}{k}\Big)+\frac{1}{k}\sin^2\theta-\frac{1}{n}\in {V}^{\top},\quad \text{for }i=1,2,\cdots, k.
\end{align*}
Similarly, we can obtain that \begin{align*}
  m_{k+j}^2-\frac{1}{n}=  \cos^2\theta\Big(\xi_j^2-\frac{1}{n-k}\Big)+\frac{1}{n-k}\cos^2\theta-\frac{1}{n}\in {V}^{\top},\quad \text{for }j=1,2,\cdots, n-k,
\end{align*}
which concludes the lemma.
\end{proof}

For any $\phi(m)\in V$, we apply the decomposition:
\begin{align*}
  \phi(m):=\phi^{\top}(m)+\phi^{\bot}(m),
\end{align*}
where $\phi^{\top}\in {V}^\top$, and $\phi^{\bot} \in {V}^{\bot}$.
As $f_0(m)$ depends only on $\theta$, one get from  $\phi^\top\in {V}^{\top}$ that
\begin{align*}
  \frac{\phi}{f_0}-\fint_{\mathbb{S}^{n-1}}\frac{\phi}{f_0}\ud m \in {V}^{\top}.
\end{align*}
Thus we obtain that
$ \int_{\mathbb{S}^{n-1}}\frac{\phi^{\top}\phi^{\bot}}{f_0} dm=0,$
which implies
\begin{align*}
  \int_{\mathbb{S}^{n-1}}\frac{\phi^2}{f_0} \ud m
  = \int_{\mathbb{S}^{n-1}}\frac{(\phi^{\top})^2}{f_0} \ud m+\int_{\mathbb{S}^{n-1}}\frac{(\phi^{\bot})^2}{f_0} \ud m.
\end{align*}
Therefore, we obtain from Lemma \ref{lem:ortho} that
\begin{align}\label{decomp:H-0}
 \langle \mathcal{H}_{f_0} \phi,\phi\rangle= \langle \mathcal{H}_{f_0} \phi^\top,\phi^\top\rangle
 +\int_{\mathbb{S}^{n-1}}\frac{(\phi^{\bot})^2}{f_0} dm\ge \langle \mathcal{H}_{f_0} \phi^\top,\phi^\top\rangle.
\end{align}

Next, we will give further decomposition for $\langle\mathcal{H}_{f_0} \phi^\top,\phi^\top\rangle$.
\begin{lemma}\label{lem:integral}
For $d\ge 2$, and $\varpi=(\varpi_i)_{1\le i\le d}\in \mathbb{S}^{d-1}$, one has
  \begin{align}\label{eq:integral}
\int_{\mathbb{S}^{d-1}} \varpi_i^2\varpi_j^2 \ud \varpi =\frac{S_d}{d(d+2)}\text{ for }i\neq j,\quad  \int_{\mathbb{S}^{d-1}} \varpi_i^4\ud \varpi =\frac{3S_d}{d(d+2)}.
\end{align}
\end{lemma}
\begin{proof}This is elementary.
First we have
\begin{align*}
\int_{\mathbb{S}^{d-1}} \varpi_i^2 \ud \varpi =\frac{1}{d}\int_{\mathbb{S}^{d-1}} \ud \varpi =\frac{S_d}{d}.
\end{align*}
Let $\tilde\varpi_1=(\varpi_1+\varpi_2)/\sqrt{2},\, \tilde\varpi_2=(\varpi_1-\varpi_2)/\sqrt{2}$. Then from the facts that
\begin{align*}
\int_{\mathbb{S}^{d-1}} \varpi_1^2\varpi_2^2 \ud \varpi=\int_{\mathbb{S}^{d-1}} \tilde\varpi_1^2\tilde\varpi_2^2 \ud \varpi = \frac12 \int_{\mathbb{S}^{d-1}} (\varpi_1^4-\varpi_1^2\varpi_2^2) \ud \varpi,\\
(d-1)\int_{\mathbb{S}^{d-1}} \varpi_1^2\varpi_2^2 \ud \varpi + \int_{\mathbb{S}^{d-1}} \varpi_1^4\ud \varpi =\int_{\mathbb{S}^{d-1}} \varpi_1^2 \ud \varpi =\frac{S_d}{d},
\end{align*}
we obtain \eqref{eq:integral}.
\end{proof}

Applying \eqref{eq:integral} with $d=k$ and $n-k$, we get the following proposition.
\begin{proposition}\label{prop:orth}
Elements in $U$ are pairwisely orthogonal under the inner product $\langle\,, \,\rangle_{\mathbb{S}^{k-1}\times\mathbb{S}^{n-k-1}}$. Indeed,
\begin{align}\label{eq:omega}
\int_{\mathbb{S}^{k-1}}\Omega_{I}(\omega)\Omega_{J}(\omega)\ud \omega=&\,\delta_{I}^J\frac{S_k}{k(k+2)},\quad \text{ for }I, J\in \Lambda_1;\\ \label{eq:xi}
\int_{\mathbb{S}^{n-k-1}}\Xi_{I}(\xi)\Xi_{J}(\xi)\ud\xi=&\,\delta_{I}^J\frac{S_{n-k}}{(n-k)(n-k+2)},\quad \text{ for }I, J\in \Lambda_2;\\\label{eq:theta}
\int_{\mathbb{S}^{k-1}\times\mathbb{S}^{n-k-1}}\Theta_{I}(\omega,\xi)\Theta_{J}(\omega,\xi)\ud\omega\ud\xi=&\,\delta_{I}^J\frac{S_{k}S_{n-k}}{k(n-k)},\quad \text{ for }I, J\in \Lambda_0,
\end{align}
and  it holds for $I\in \Lambda_1$, $J\in \Lambda_2$, $K\in \Lambda_0$ that
\begin{align}
\label{eq:cross}
\int_{\mathbb{S}^{k-1}\times\mathbb{S}^{n-k-1}}\Big\{\Omega_{I}(\omega)\Theta_{K}(\omega,\xi),\  \Xi_{J}(\xi)\Theta_{K}(\omega,\xi),\ \Omega_{I}(\omega)\Xi_{J}(\xi)\Big\}\ud\omega\ud\xi=&\,0.%,\quad \text{ for }
%\int_{\mathbb{S}^{k-1}\times\mathbb{S}^{n-k-1}}\Omega_{I}(\omega)\Theta_{K}(\omega,\xi)\ud\omega\ud\xi=&\,0,\quad \text{ for }I\in \Lambda_1,\  K\in \Lambda_0;\\
%\int_{\mathbb{S}^{k-1}\times\mathbb{S}^{n-k-1}}\Xi_{J}(\xi)\Theta_{K}(\omega,\xi)\ud\omega\ud\xi=&\,0,\quad \text{ for }J\in \Lambda_2,\  K\in \Lambda_0;\\
%\int_{\mathbb{S}^{k-1}\times\mathbb{S}^{n-k-1}}\Omega_{I}(\omega)\Xi_{J}(\xi)\ud\omega\ud\xi=&\,0,\quad \text{ for }I\in \Lambda_1,\  J\in \Lambda_2.
\end{align}
\end{proposition}

Define
\begin{align}\nonumber
W_{i}(\varphi)=&\,\int_{\mathbb{S}^{k-1}}\left(\omega_i^2-\frac{1}{k}\right)\varphi(\omega) \ud \omega,\quad \text{ for $1\le i\le k$, $\varphi\in L^2(\mathbb{S}^{k-1})$},\\
\nonumber X_{j}(\psi)=&\,\int_{\mathbb{S}^{n-k-1}}\left(\xi_j^2-\frac{1}{n-k}\right)\psi(\xi) \ud \xi,\quad \text{ for $1\le j\le n-k$, $\psi\in L^2(\mathbb{S}^{n-k-1})$}.
\end{align}

\begin{proposition}\label{prop:decomp}
For $I\in \Lambda_1^B$, $J\in \Lambda_2^B$, there holds
\begin{align}\label{eq:WX-B}
W_{i}(\Omega_{I})=0,\ \text{for }I\in\Lambda_1^B;  \quad  X_{j}(\Xi_{J})=0,\ \text{for }J\in\Lambda_2^B.
\end{align}
Moreover, for  $I, I'\in \Lambda_1^A$, $J, J'\in \Lambda_2^A$,
\begin{align}\label{eq:W}
\sum_{i=1}^k W_{i}(\Omega_{I})=0,\quad& \sum_{i=1}^k W_{i}(\Omega_{I})W_{i}(\Omega_{I'})=\frac{2S_k^2}{k^2(k+2)^2}\delta_I^{I'};\\
\sum_{j=1}^{n-k} X_{j}(\Xi_{J})=0,\quad& \sum_{j=1}^{n-k} X_{j}(\Xi_{I})X_{j}(\Xi_{J'})=\frac{2S_{n-k}^2}{(n-k)^2(n-k+2)^2}\delta_J^{J'}.\label{eq:X}
\end{align}
\end{proposition}
\begin{proof}
Equations (\ref{eq:WX-B}) follows directly from the definitions.
To prove \eqref{eq:W}, we may assume $I=0l, I'=0l'$.  From  Lemma \ref{lem:integral}, we have
\begin{equation}
\sqrt{2l(l+1)}W_{i}(\Omega_{0l})=\int_{\mathbb{S}^{k-1}}\left(\omega_i^2-\frac{1}{k}\right)\left(l\omega_{l+1}^2-\sum_{j=1}^l\omega_j^2\right) \ud\omega
=\frac{2S_k}{k(k+2)}\times \left\{\begin{aligned}
&-1,& \text{for }1\le i\le l;\\
&l, &\text{for }i=l+1;\\
&0,&\text{for }l+1<i\le k,
\end{aligned}\right.\nonumber
\end{equation}
which yields (\ref{eq:W}) by direct calculations. The proof of (\ref{eq:X}) is similar.
\end{proof}

We introduce the functionals on $L^2\big([0,\pi/2],\ud\mu_k^{\theta}\big)$:
\begin{align*}
\mathcal{I}_0(a)=&\Big(\int_0^{\frac{\pi}{2}}e^{\eta\sin^2\theta} \ud\mu_k^{\theta}\Big)\Big(\int_0^{\frac{\pi}{2}}e^{-\eta\sin^2\theta}a^2(\theta)\ud\mu_k^{\theta}\Big)
    -\frac{2\alpha}{k(n-k)}\left(\int_{0}^{\frac\pi2}\sin\theta\cos\theta a(\theta)\ud\mu_k^{\theta} \right)^2;\\
\mathcal{I}_1(a)=&\Big(\int_0^{\frac{\pi}{2}}e^{\eta\sin^2\theta} \ud\mu_k^{\theta}\Big)\Big(\int_0^{\frac{\pi}{2}}e^{-\eta\sin^2\theta}a^2(\theta)\ud\mu_k^{\theta}\Big)
-\frac{2\alpha}{k(k+2)}\left(\int_{0}^{\frac\pi2}\sin^2\theta a(\theta)\ud\mu_k^{\theta} \right)^2;\\
\mathcal{I}_2(a)=&\Big(\int_0^{\frac{\pi}{2}}e^{\eta\sin^2\theta} \ud\mu_k^{\theta}\Big)\Big(\int_0^{\frac{\pi}{2}}e^{-\eta\sin^2\theta}a^2(\theta)\ud\mu_k^{\theta}\Big)
  -\frac{2\alpha}{(n-k)(n-k+2)}\left(\int_{0}^{\frac\pi2}\cos^2\theta a(\theta)\ud\mu_k^{\theta} \right)^2;\\
\mathcal{I}_3(b)=&\Big(\int_0^{\frac{\pi}{2}}e^{\eta\sin^2\theta} \ud\mu_k^{\theta}\Big)\Big(\int_0^{\frac{\pi}{2}}e^{-\eta\sin^2\theta}b^2(\theta)\ud\mu_k^{\theta}\Big)
  -\frac{n\alpha}{k(n-k)}\left(\int_{0}^{\frac\pi2}\sin^2\theta b(\theta)\ud\mu_k^{\theta} \right)^2.
\end{align*}
The following diagonalization lemma, together with \eqref{decomp:H-0}, enable us to reduce the analysis of $\langle \mathcal{H}_{f_0}\phi,\phi\rangle $ to the study of the above four functionals.
\begin{lemma}\label{lem:decomp}
For $\phi\in {V}^{\top}$ be given by
\begin{equation}\label{decomp:phi}
\begin{aligned}
\phi(\theta,\omega,\xi)=&\sum_{I\in\Lambda_1}a_{1}^I(\theta)\Omega_{I}(\omega)+\sum_{J\in\Lambda_2}a_{2}^J(\theta)\Xi_{J}(\xi)
+\sum_{K\in \Lambda_0}a_{0}^K(\theta)\Theta_{K}(\omega,\xi)+b(\theta),
\end{aligned}
\end{equation}
we have the following decomposition:
\begin{align}\nonumber
\langle \mathcal{H}_{f_0}\phi,\phi\rangle =&\,\frac{S_{k}^2S_{n-k}^2}{k(k+2)} \sum_{I\in\Lambda_1}\mathcal{I}_1(a_{1}^I)+
\frac{S_{k}^2S_{n-k}^2}{(n-k)(n-k+2)} \sum_{J\in\Lambda_2}\mathcal{I}_2(a_{2}^J)\\ \label{decomp:H}
&+\frac{S_{k}^2S_{n-k}^2}{k(n-k)} \sum_{K\in\Lambda_0}\mathcal{I}_0(a_{0}^K)+ S_{k}^2S_{n-k}^2\mathcal{I}_3(b).
\end{align}
\end{lemma}
\begin{remark}
  When $k=1$ (or $n-1$), there is no $\mathcal{I}_1$ (or $\mathcal{I}_2$) term in (\ref{decomp:phi})-(\ref{decomp:H}) as $\Lambda_1$(or $\Lambda_2$) $=\emptyset$.
\end{remark}
\begin{proof}
For $\phi(m)=\phi(\theta,\omega,\xi)$ takes the form (\ref{decomp:phi}), we can write
\begin{align*}
\langle \mathcal{H}_{f_0}\phi,\phi\rangle=&\int_{\mathbb{S}^{n-1}}\frac{\phi^2}{f_0}\ud m-2\alpha\sum_{i=1}^k\sum_{j=1}^{n-k}\left(\int_{\mathbb{S}^{n-1}}\sin\theta\cos\theta \omega_i\xi_j\phi(m)\ud m \right)^2\\
&\quad -\alpha\sum_{i=1}^k\left(\int_{\mathbb{S}^{n-1}}\sin^2\theta\omega_i^2\phi(m)\ud m\right)^2-2\alpha\sum_{1\le i<i'\le k}\left(\int_{\mathbb{S}^{n-1}}\sin^2\theta \omega_i\omega_{i'}\phi(m)\ud m\right)^2\\
&\quad-\alpha\sum_{j=1}^{n-k}\left(\int_{\mathbb{S}^{n-1}}\cos^2\theta\xi_j^2\phi(m)\ud m\right)^2-2\alpha\sum_{1\le j<j'\le n-k}\left(\int_{\mathbb{S}^{n-1}}\cos^2\theta\xi_j\xi_{j'}\phi(m)\ud m\right)^2.%\\
\end{align*}
First, using (\ref{eq:theta}) we obtain
\begin{align*}
\sum_{i=1}^k\sum_{j=1}^{n-k}\left(\int_{\mathbb{S}^{n-1}}\sin\theta\cos\theta \omega_i\xi_j\phi(m)\ud m \right)^2
&=  \sum_{K\in\Lambda_0}\left(\int_{\mathbb{S}^{n-1}}\sin\theta\cos\theta \Theta_K(\omega)\phi(m)\ud m\right)^2\\
&= \frac{S_{k}^2S_{n-k}^2}{k^2(n-k)^2} \sum_{K\in\Lambda_0}\left(\int_{0}^{\frac{\pi}{2}}\sin\theta\cos\theta a_0^K(\theta)\ud\mu_k^{\theta} \right)^2,
\end{align*}
where we have used \eqref{decomp:measure}.
Applying Propositions \ref{prop:orth} and \ref{prop:decomp}, we have
\begin{align*}
&\sum_{i=1}^k\left(\int_{\mathbb{S}^{n-1}}\sin^2\theta\omega_i^2 \phi(m)\ud m\right)^2\\
=&\sum_{i=1}^k\left(\int_{\mathbb{S}^{n-1}}\sin^2\theta\omega_i^2\Big[\sum_{I\in\Lambda_1}a_{1}^I(\theta)\Omega_{I}(\omega)+b(
\theta)\Big]\ud m\right)^2\\
=&\,\sum_{i=1}^k\left(\sum_{I\in\Lambda_1}S_{n-k} W_i(\Omega_{I})\int_{0}^{\frac\pi2}\sin^2\theta a_{1}^I(\theta)\ud\mu_k^{\theta} +\frac{S_kS_{n-k}}{k}\int_{0}^{\frac\pi2}\sin^2\theta b(\theta)\ud \mu_k^{\theta}\right)^2\\
=&\,S_{n-k}^2\sum_{I\in\Lambda_1^A}\sum_{i=1}^kW_i^2(\Omega_{I})\left(\int_{0}^{\frac\pi2}\sin^2\theta a_{1}^I(\theta)\ud\mu_k^{\theta} \right)^2 +\frac{S_k^2S_{n-k}^2}{k}\left(\int_{0}^{\frac\pi2}\sin^2\theta b(\theta)\ud \mu_k^{\theta}\right)^2\\
=&\, \frac{2 S_k^2S_{n-k}^2}{k^2(k+2)^2}\sum_{I\in\Lambda_1^A}
\left(\int_{0}^{\frac\pi2}\sin^2\theta a_{1}^I(\theta)\ud\mu_k^{\theta} \right)^2
+\frac{S_k^2S_{n-k}^2}{k}\left(\int_{0}^{\frac\pi2}\sin^2\theta b(\theta)\ud \mu_k^{\theta}\right)^2.
\end{align*}
Similarly,
\begin{align*}
&  \sum_{j=1}^{n-k}\left(\int_{\mathbb{S}^{n-1}}\cos^2\theta\xi_j^2\phi(m)\ud m\right)^2\\
  =&\, \frac{2 S_k^2S_{n-k}^2}{(n-k)^2(n-k+2)^2}\sum_{I\in\Lambda_2^A}
\left(\int_{0}^{\frac\pi2}\cos^2\theta a_{2}^I(\theta)\ud\mu_k^{\theta} \right)^2
+\frac{S_k^2S_{n-k}^2}{n-k}\left(\int_{0}^{\frac\pi2}\cos^2\theta b(\theta)\ud \mu_k^{\theta}\right)^2.
\end{align*}
We get from (\ref{eq:omega}) that
\begin{align*}
  \sum_{1\le i<i'\le k}\left(\int_{\mathbb{S}^{n-1}}\sin^2\theta \omega_i\omega_{i'}\phi(m)\ud m\right)^2
&\,  =\sum_{I\in\Lambda_1^B}\left(\int_{\mathbb{S}^{n-1}}\sin^2\theta \Omega_I(\omega)\phi(m)\ud m\right)^2\\
&\,  =\frac{S_k^2S_{n-k}^2}{k^2(k+2)^2}\sum_{I\in\Lambda_1^B}\left(\int_{0}^{\frac{\pi}{2}}\sin^2\theta a_1^I(\theta)\ud \mu_k^{\theta}\right)^2.
\end{align*}
By similar calculation, we have
\begin{align*}
  \sum_{1\le j<j'\le n-k}\left(\int_{\mathbb{S}^{n-1}}\cos^2\theta\Xi_{jj'}\phi(m)\ud m\right)^2
&\,  =\sum_{J\in\Lambda_2^B}\left(\int_{\mathbb{S}^{n-1}}\cos^2\theta \Omega_I(\omega)\phi(m)\ud m\right)^2\\
&\,  =\frac{S_k^2S_{n-k}^2}{(n-k)^2(n-k+2)^2}\sum_{J\in\Lambda_2^B}
\left(\int_{0}^{\frac{\pi}{2}}\cos^2\theta a_2^J(\theta)\ud \mu_k^{\theta}\right)^2.
\end{align*}

Note that $\frac{1}{f_0}=S_kS_{n-k}e^{-\eta\sin^2\theta}\int_0^{\frac{\pi}{2}}e^{\eta\sin^2\theta} \ud\mu_k^{\theta}$. Moreover, Proposition \ref{prop:orth}
gives us that the elements in $U$ are pairwisely orthogonal under the $L^2(\mathbb{S}^{k-1}\times \mathbb{S}^{n-k-1})$-inner product. Thus we obtain
\begin{align*}
\int_{\mathbb{S}^{n-1}}\frac{\phi^2}{f_0}\ud m=
&\int_{\mathbb{S}^{n-1}}\frac{1}{f_0}\Big(\sum_{I\in\Lambda_1}\big(a_{1}^I(\theta)\Omega_{I}(\omega)\big)^2
+\sum_{J\in\Lambda_2}\big(a_{2}^J(\theta)\Xi_{J}(\xi)\big)^2\\
&\qquad+\sum_{K\in \Lambda_0}\big(a_{0}^K(\theta)\Theta_{K}(\omega,\xi)\big)^2+b^2(\theta)\Big)\ud m\\
=&\,S_k^2S_{n-k}^2\Big(\int_0^{\frac{\pi}{2}}e^{\eta\sin^2\theta} \ud\mu_k^{\theta}\Big)\int_0^{\frac{\pi}{2}}e^{-\eta\sin^2\theta}\bigg\{\frac{1}{k(k+2)}\sum_{I\in\Lambda_1}\big(a_{1}^I(\theta)\big)^2
\\
&\qquad+\frac{1}{(n-k)(n-k+2)}\sum_{J\in\Lambda_2}\big(a_{2}^J(\theta)\big)^2+\frac{1}{k(n-k)}\sum_{K\in \Lambda_0}\big(a_{0}^K(\theta)\big)^2+b^2(\theta)\bigg\}\ud\mu_k^{\theta}.
\end{align*}
Combining the above equalities, we obtain the decomposition \eqref{decomp:H}.
\end{proof}

\section{Nonnegativity of decomposed functionals}\label{sec:nonnegativity}
By Lemma $\ref{lem:decomp}$, it suffices to study the nonnegativity of functionals $\mathcal{I}_\gamma(\gamma=0,1,2,3)$.
Define
\begin{align*}
  A_l(\eta):=\int_0^{\frac{\pi}{2}}e^{\eta\sin^2\theta}\sin^l\theta \ud\mu_k^{\theta}.
\end{align*} Then
\begin{align}\label{eq:alpha}
  \alpha=\sigma_k(\eta)=\frac{(n-k)k}{2}\frac{A_0}{(A_2-A_4)}.
\end{align}
\begin{proposition}\label{prop:I0}
  $\mathcal{I}_0(a)\ge 0$ for all $a\in L^2\big([0,\pi/2],\ud\mu_k^{\theta}\big)$ .
\end{proposition}
\begin{proof}
We apply the Cauchy-Schwarz inequality to obtain:
\begin{align*}
\mathcal{I}_0(a)
 =&\, \Big(\int_0^{\frac{\pi}{2}} e^{\eta\sin^2\theta} \ud\mu_k^{\theta}\Big)\Big(\int_0^{\frac{\pi}{2}}e^{-\eta\sin^2\theta} a^2(\theta) \ud\mu_k^{\theta}\Big)- \frac{2\alpha}{k(n-k)}\Big(\int_0^{\frac{\pi}{2}}\sin\theta\cos\theta a(\theta) \ud\mu_k^{\theta}\Big)^2\\
\geq &\, \Big(\int_0^{\frac{\pi}{2}}e^{-\eta\sin^2\theta}a^2(\theta)\ud \mu(\theta) \Big)\left(\int_0^{\frac{\pi}{2}}e^{\eta\sin^2\theta}\ud\mu_k^{\theta}-
\frac{2\alpha}{k(n-k)}\int_0^{\frac{\pi}{2}}e^{\eta\sin^2\theta}\sin^2\theta\cos^2\theta\ud\mu_k^{\theta}\right)\\
=&~0.
\end{align*}
Here we used the relation $\alpha=\sigma_k(\eta)$. The equality holds if and only if $a(\theta)=Ce^{\eta\sin^2\theta}\sin\theta\cos\theta$.
\end{proof}
\begin{proposition}\label{prop:I1}
  $\mathcal{I}_1(a)\ge 0$ for all $a\in L^2\big([0,\pi/2],\ud\mu_k^{\theta}\big)$ if and only if $\eta\le 0$.
\end{proposition}
\begin{proof}
Using Cauchy-Schwarz inequality, we have
\begin{align*}
\mathcal{I}_1(a)=&\,\Big(\int_0^{\frac{\pi}{2}} e^{\eta\sin^2\theta} \ud\mu_k^{\theta}\Big)\Big( \int_0^{\frac{\pi}{2}}e^{-\eta\sin^2\theta} a^2(\theta) \ud\mu_k^{\theta}\Big)-\frac{2\alpha}{k(k+2)} \Big(\int_0^{\frac{\pi}{2}}\sin^2\theta a(\theta) \ud\mu_k^{\theta} \Big)^2\\
\geq&\,\Big(\int_0^{\frac{\pi}{2}}e^{-\eta\sin^2\theta} a^2(\theta)\ud\mu_k^{\theta}\Big)
\left(\int_0^{\frac{\pi}{2}}e^{\eta\sin^2\theta}\ud\mu_k^{\theta}
-\frac{2\alpha}{k(k+2)}\int_0^{\frac{\pi}{2}}e^{\eta\sin^2\theta}\sin^4\theta \ud\mu_k^{\theta}\right).
\end{align*}
Note that the equality can be attained by $a(\theta)=Ce^{\eta\sin^2\theta}\sin^2\theta$. Applying $\alpha=\sigma_k(\eta)$, we get
\begin{align*}
&\int_0^{\frac{\pi}{2}}e^{\eta\sin^2\theta}\ud\mu_k^{\theta}-\frac{2\alpha}{k(k+2)}\int_0^{\frac{\pi}{2}}e^{\eta\sin^2\theta}\sin^4\theta \ud\mu_k^{\theta}\\
 &=\frac{2\alpha}{k(k+2)(n-k)}\int_0^{\frac{\pi}{2}}e^{\eta\sin^2\theta}\sin^2\theta\Big[(k+2)-(n+2)\sin^2\theta\Big]\ud\mu_k^{\theta}\\
&=\frac{2\alpha}{k(k+2)(n-k)} \int_0^{\frac{\pi}{2}} \bigg(e^{\eta\sin^2\theta}-e^{\eta\frac{k+2}{n+2}}\bigg)\sin^2\theta\Big[(k+2)-(n+2)\sin^2\theta\Big]\ud\mu_k^{\theta},
\end{align*}
which is nonnegative if and only if $\eta\le 0$.
\end{proof}

\begin{proposition}\label{prop:I2}
  $\mathcal{I}_2(a)\ge 0$ for all $a\in L^2\big([0,\pi/2],\ud\mu_k^{\theta}\big)$ if and only if $\eta\ge 0$.
\end{proposition}
\begin{proof}
Using Cauchy-Schwarz inequality, we have
\begin{align*}
\mathcal{I}_2(a)\geq&\,\Big(\int_0^{\frac{\pi}{2}}e^{-\eta\sin^2\theta} a^2(\theta)\ud\mu_k^{\theta}\Big)\left(\int_0^{\frac{\pi}{2}}e^{\eta\sin^2\theta}\ud\mu_k^{\theta}
-\frac{2\alpha}{(n-k)(n-k+2)}\int_0^{\frac{\pi}{2}}e^{\eta\sin^2\theta}\cos^4\theta \ud\mu_k^{\theta}\right).
\end{align*}
Applying $\alpha=\sigma_k(\eta)$ similarly, we
\begin{align*}
&\int_0^{\frac{\pi}{2}}e^{\eta\sin^2\theta}\ud\mu_k^{\theta}-\frac{2\alpha}{(n-k)(n-k+2)}\int_0^{\frac{\pi}{2}}e^{\eta\sin^2\theta}\cos^4\theta \ud\mu_k^{\theta}\\
 &=\frac{2\alpha}{k(n-k+2)(n-k)}\int_0^{\frac{\pi}{2}}e^{\eta\sin^2\theta}\cos^2\theta\Big[(n+2)\sin^2\theta-k\Big]\ud\mu_k^{\theta}\\
&=\frac{2\alpha}{k(n-k+2)(n-k)} \int_0^{\frac{\pi}{2}} \Big(e^{\eta\sin^2\theta}-e^{\eta\frac{k}{n+2}}\Big)\cos^2\theta\Big[(n+2)\sin^2\theta-k\Big]\ud\mu_k^{\theta},
\end{align*}
which is nonnegative if and only if $\eta\ge 0$.
\end{proof}

\begin{proposition}\label{prop:I3}
  $\mathcal{I}_3(b)\ge 0$ for all $b\in L^2\big([0,\pi/2],\ud\mu_k^{\theta}\big)$ with $\int_0^{\frac{\pi}{2}}b(\theta)\ud\mu_k^{\theta}=0$,
  if and only if $$2A_0(A_2-A_4)-n(A_0A_4-{A_2^2})\ge 0.$$
\end{proposition}
\begin{proof}
As $\int_{-\frac{\pi}{2}}^{\frac{\pi}{2}}b(\theta)\ud\mu_k^{\theta}=0$, we have
\begin{align*}
n\bigg(\int_0^{\frac{\pi}{2}}\sin^2\theta b(\theta)\ud\mu_k^{\theta}\bigg)^2=&\,n\bigg(\int_0^{\frac{\pi}{2}}\Big(\frac{A_2}{A_0}-\sin^2\theta\Big) b(\theta)\ud\mu_k^{\theta}\bigg)^2\\
\le &\,n\int_0^{\frac{\pi}{2}}e^{\eta\sin^2\theta}\Big(\frac{A_2}{A_0}-\sin^2\theta\Big)^2 \ud\mu_k^{\theta}\cdot\int_0^{\frac{\pi}{2}}e^{-\eta\sin^2\theta} b^2(\theta) \ud\mu_k^{\theta} \\
=&\,n\Big(A_4-\frac{A_2^2}{A_0}\Big)\int_0^{\frac{\pi}{2}}e^{-\eta\sin^2\theta} b^2(\theta) \ud\mu_k^{\theta}.
\end{align*}
The equality can be achieved by $b(\theta)=Ce^{\eta\sin^2\theta}({A_2}/{A_0}-\sin^2\theta)$.  Thus,
\begin{align*}
\mathcal{I}_2(b)
&\ge \bigg(A_0-\frac{n\alpha}{(n-k)k}(A_4-\frac{A_2^2}{A_0})\bigg)\int_0^{\frac{\pi}{2}}e^{-\eta\sin^2\theta} b^2(\theta) \ud\mu_k^{\theta}\\
&= \bigg(A_0-\frac{n(A_0A_4-{A_2^2})}{2(A_2-A_4)}\bigg)\int_0^{\frac{\pi}{2}}e^{-\eta\sin^2\theta} b^2(\theta) \ud\mu_k^{\theta},
\end{align*}
which concludes the proof.
\end{proof}
It remains to study the sign of $2A_0(A_2-A_4)-n(A_0A_4-{A_2^2})$ for different $\eta$. This is summarized in Lemma \ref{lem:eta-value}.

%We first introduces some important properties for the functions $\sigma_k(\eta)$.
\begin{lemma}\label{lem1}
It holds that
\begin{align}\label{relation:A}
A_{l+2}-A_{l+4}&\,=\frac{n+l}{2\eta}A_{l+2}-\frac{k+l}{2\eta}A_{l}.
\end{align}

\end{lemma}
\begin{proof}We apply the integration by part to obtain
\begin{align*}
 0=& \int_0^{\frac{\pi}{2}}\ud( e^{\eta\sin^2\theta}\sin^{k+l}\theta \cos^{n-k}\theta )\\
 =& \int_0^{\frac{\pi}{2}} e^{\eta\sin^2\theta}\sin^{k+l-1}\theta \cos^{n-k-1}\theta \Big[(k+l)\cos^2\theta -(n-k)\sin^2\theta\Big]\ud \theta\\
&+2\eta\int_0^{\frac{\pi}{2}} e^{\eta\sin^2\theta}\sin^{k+l+1}\theta \cos^{n-k+1}\theta \ud\theta,
\end{align*}
which yields \eqref{relation:A}.
\end{proof}

\begin{lemma}\label{lem:single-root}
(i) It holds that
\begin{align}\label{lim:pinf}
  \lim_{\eta\to+\infty}\sigma_k'(\eta)=& \lim_{\eta\to+\infty}\frac{\sigma_k}{\eta} = k,\\ \label{lim:minf}
  \lim_{\eta\to-\infty}\sigma_k'(\eta)=& \lim_{\eta\to-\infty}\frac{\sigma_k}{\eta} = k-n.
\end{align}
(ii) There is a unique $\eta_k^*$ such that $\sigma_k(\eta_k^*)=\min_{\eta} \sigma_k(\eta)$. Thus, $\sigma_k'(\eta)$ has the same sign with $\eta-\eta^*_k$.
\end{lemma}
\begin{proof}
It is not hard to prove that $\lim_{\eta\to +\infty} \frac{A_2}{A_0} =1$. Thus by (\ref{eq:alpha}) and Lemma \ref{lem1}, we have
\begin{align*}
 \frac{\sigma_k}{\eta}=\frac{k(n-k)A_0}{2\eta(A_2-A_4)} =\frac{k(n-k)A_0}{nA_2-kA_0} \to k,\quad \text{as }\eta\to+\infty,
\end{align*}
which yields \eqref{lim:pinf}. Similarly, \eqref{lim:minf} follows from the fact that $\lim_{\eta\to -\infty} \frac{A_2}{A_0} =0$.
Thus the minimum of $\sigma_k(\eta)$ can be attained, and $\sigma_k'(\eta)=0$ has at least one solution $\eta_k^*$. It suffices to prove it is unique.

On the other hand, from (\ref{eq:alpha}) and (\ref{relation:A}) we obtain
\begin{align}\nonumber
\sigma_k'(\eta)=\frac{(n-k)k\big(A_2(A_2-A_4)-A_0(A_4-A_6)\big)}{2(A_4-A_2)^2}.
\end{align}
Using Cauchy-Schwarz inequality, we have
\begin{align}\nonumber
&\frac{\partial }{\partial \eta }\Big(e^{-\eta}\big(A_2(A_2-A_4)-A_0(A_4-A_6)\big)\Big)\\ \nonumber
=&\,e^{-\eta}\Big(-(A_2-A_4)^2+A_0(A_4-2A_6+A_8)\Big)\\\nonumber
=&-e^{-\eta}\left(\int_0^{\frac{\pi}{2}} e^{\eta\sin^2\theta} \sin^2\theta(1-\sin^2\theta) \ud\mu_k^{\theta}\right)^2\\
&+e^{-\eta}\int_0^{\frac{\pi}{2}} e^{\eta\sin^2\theta} \ud\mu_k^{\theta}\cdot\int_0^{\frac{\pi}{2}} e^{\eta\sin^2\theta} \sin^4\theta(1-\sin^2\theta)^2\ud\mu_k^{\theta}>0.
 \label{def-sigprime}
\end{align}
Hence, $A_2(A_2-A_4)-A_0(A_4-A_6)$ has only one root, which implies $\sigma_k'(\eta)=0$ has only one root $\eta_k^*$.
\end{proof}

\begin{lemma}\label{lem3}
For $k< (\text{or }>, =)[\frac{n}{2}]$, we have $\eta_k^{*}>(\text{or }<, =)0$.
\end{lemma}
\begin{proof}
From Lemma \ref{lem:single-root}, we get that $\eta-\eta^*_k$ has the same sign with $\sigma_k'(\eta)$. Thus, $\eta^*_k>0$ if and only if $\sigma_k'(0)<0$. On the other hand, we have
\begin{align*}
\frac{2(A_4-A_2)^2(0)}{(n-k)k}\sigma_k'(0)&\,=\Big(A_2(A_2-A_4)-A_0(A_4-A_6)\Big)\Big|_{\eta=0}\\
&\,=\Big(\int_0^{\frac{\pi}{2}}\sin^{k+1}\theta \cos^{n-k-1}\theta \ud\theta\Big)\Big(\int_0^{\frac{\pi}{2}}\sin^{k+1}\theta\cos^{n-k+1}\theta \ud\theta\Big)\\
&\,\quad-\Big(\int_0^{\frac{\pi}{2}}\sin^{k-1}\theta\cos^{n-k-1}\theta \ud\theta\Big)\Big(\int_0^{\frac{\pi}{2}} \sin^{k+3}\theta \cos^{n-k+1}\theta \ud\theta\Big)\\
&\,=\frac{2(2k-n)}{(k+2)k}\Big(\int_0^{\frac{\pi}{2}}\sin^{k+1}\theta \cos^{n-k-1}\theta \ud\theta\Big) \Big(\int_0^{\frac{\pi}{2}} \sin^{k+3}\theta \cos^{n-k+1}\theta \ud\theta\Big).
\end{align*}
Here we have used the facts that
\begin{align*}
k\int_0^{\frac{\pi}{2}} \sin^{k-1}\theta\cos^{n-k-1}\theta \ud\theta - n\int_0^{\frac{\pi}{2}}\sin^{k+1}\theta\cos^{n-k-1}\theta& \ud\theta=\int_0^{\frac{\pi}{2}}\ud (\sin^k\theta \cos^{n-k}\theta )=0,\\
(k+2)\int_0^{\frac{\pi}{2}} \sin^{k+1}\theta\cos^{n-k+1}\theta \ud\theta - (n+4)\int_0^{\frac{\pi}{2}}\sin^{k+3}\theta&\cos^{n-k+1}\theta \ud\theta\\
=\,&\int_0^{\frac{\pi}{2}}\ud (\sin^{k+2}\theta \cos^{n-k+2}\theta )=0.
\end{align*}
Thus, $\eta_k^{*}>(\text{or }<, =)0$ if and only if $k<(\text{or }>, =)[\frac{n}{2}]$.
\end{proof}

\begin{lemma}\label{lem:eta-value}
The inequality $n(A_0A_4-A_2^2)<2A_0(A_2-A_4)$ is equivalent to $\eta(\eta-\eta^{*}_k)>0$. When $k\le [\frac{n}{2}]$, this is ensured by $\eta>\eta_k^*$ as $\eta_k^*>0$.
\end{lemma}
\begin{proof}
We obtain from Lemma \ref{lem1} that
\begin{align*}
&A_4-A_6=\frac{n+2}{2\eta}A_4-\frac{k+2}{2\eta}A_2,\quad A_2-A_4=\frac{n}{2\eta}A_2-\frac{k}{2\eta}A_0.
\end{align*}
Hence
\begin{align*}
2A_0(A_2-A_4)-n(A_0A_4-A_2^2) ={2\eta}\Big(A_2(A_2-A_4)-A_0(A_4-A_6)\Big)>0,
\end{align*}
which concludes the proof by recalling \eqref{def-sigprime} and Lemma \ref{lem:single-root}.
\end{proof}

\section{Proof of the main theorems}\label{sec:proof}

\subsection{Proof of Theorem \ref{thm:1}}
We begin with anisotropic critical points $f_0=h_\eta^{(k)}(m)$ with $\eta\neq 0$. In this case $\alpha=\sigma_k(\eta)$.

If  $2\le k\le [\frac{n}{2}]$, by Propositions \ref{prop:I1} and \ref{prop:I2}, we
know that $\mathcal{I}_1$ and $\mathcal{I}_2$ can not be always nonnegative. Then we may choose suitable perturbation $\phi=a(\theta)\Omega(\omega)$ or $a(\theta)\Xi(\xi)$ such that $\langle \mathcal{H}_{f_0}\phi,\phi\rangle <0$, which implies the claim $(ii)$ of Theorem \ref{thm:1}.

When $k=1$, there is no $\mathcal{I}_1$ term in \eqref{decomp:H}. We know from  Propositions \ref{prop:I0}, \ref{prop:I2}, \ref{prop:I3},
and Lemma \ref{lem:eta-value} that, for $\eta>\eta_1^*$, all functionals
$\mathcal{I}_\gamma(\gamma=0,2,3)$ are nonnegative. Thus, $h_{\eta}^{(k)}$ is stable. When $\eta<\eta_1^*$, one has $\eta<0$ or $\eta(\eta-\eta_1^*)<0$. So, we may find $a(\theta)$ or $b(\theta)$ such that $\mathcal{I}_2(a)$ or $\mathcal{I}_3(b)$ is negative, which implies that the equilibria is not stable. This gives the claim $(iii)$ of Theorem \ref{thm:1}.

\medskip

For the claim $(i)$ of Theorem \ref{thm:1}, we note that, the stability of the isotropic equilibria $h_0=1/|\mathbb{S}^{n-1}|$ is indeed equivalent to the nonnegativity of $\mathcal{I}_1(a)$ with $k=n$ and $\eta=0$. Therefore we can directly
deduce from the proof of Proposition \ref{prop:I1} that it is equivalent to $\alpha<\frac{n(n+2)}{2}$.

\subsection{Proof of Theorem \ref{thm:2}}
We assume $\nu=(1,0,\cdots, 0)$.
For $\eta>\eta_1^*$, from the proofs of Propositions \ref{prop:I2} and \ref{prop:I3},
we get that $\mathcal{I}_2(a), \mathcal{I}_3(a)\ge c_0 \|a\|_{L^2(\ud\mu_k^{\theta})}^2$ for some $c_0=c_0(\eta,n)>0$. From the proof of Propositions \ref{prop:I0}, we know that $\mathcal{I}_0(a_0^K)=0$ if and only if $a_0^K(\theta)=Ce^{\eta\sin^2\theta}\sin\theta\cos\theta$, i.e.,
\begin{align*}
\phi=e^{\eta\sin^2\theta}\sin\theta\cos\theta\sum_{k=1}^{n-1} C_k\xi_k\in V_{\nu}^\top.
\end{align*}
Thus by Proposition \ref{prop:decomp}, $\langle \mathcal{H}_{f_0}\phi,\phi\rangle =0$ if and only if $\phi \in V_{\nu}^\top$.
It is also not hard to see that  $\mathcal{I}_0(a)\ge c_0 \|a\|_{L^2(\ud\mu_k^{\theta})}^2$ providing $\langle a, e^{\eta\sin^2\theta}\sin\theta\cos\theta\rangle_{L^2(\ud\mu_k^{\theta})}=0$.
Thus $\langle \mathcal{H}_{f_0}\phi,\phi\rangle \ge c_0\|\phi\|^2_{L^2(\mathbb{S}^{n-1})}$ for $\phi \in V_{\nu}^\bot$. The proof is completed.

\section*{Appendix}
The following proposition is first proved by Wang-Hoffman \cite{WH}. Here, we give another proof based on the method by J. Ball \cite{Bal}.
\begin{proposition}
If a symmetric trace free matrix $M$ satisfies:
\begin{align}\label{eq:M1}
 \frac{M}{\alpha} =\frac{\int_{\mathbb{S}^{n-1}}\Big(m\otimes m-\frac{1}{n}I_n\Big)e^{M:(m\otimes m)}\ud m}{
  \int_{\mathbb{S}^{n-1}}e^{M:(m\otimes m)}\ud m},
\end{align}
then $M$ has only two distinct eigenvalues $\lambda_1$ and $\lambda_2$.
\end{proposition}
\begin{proof}
Without loss of generality, we may assume that $M=\mathrm{diag}\,\{\lambda_1,\lambda_2,\cdots,\lambda_n\}$ with $\lambda_1,\, \lambda_2,\, \lambda_3$ being different to each other. Let $\rho(x)=e^{\sum_{i=1}\lambda_1x_i^2}$ for $|x|\le 1$. Then applying the divergence theorem we get:
\begin{align*}
0=&\int_{|x|\le 1}\Big( \partial_1\partial_3(x_1x_3\rho(x))-\partial_3\partial_1(x_1x_3\rho(x))\Big)\ud x\\
=&\int_{\mathbb{S}^{n-1}}\Big(m_1\partial_3(m_1m_3\rho(m))-m_3\partial_1(m_1m_3\rho(m))\Big)\ud m\\
=&\int_{\mathbb{S}^{n-1}}\Big((m_1^2-m_3^2)+2(\lambda_3-\lambda_1)m_1^2m_3^2\Big)\rho(m)\ud m.
\end{align*}
Let $\beta=\alpha^{-1}\int_{\mathbb{S}^{n-1}}e^{M:(m\otimes m)}\ud m$. Then we obtain from \eqref{eq:M1} that
\begin{align*}
  \beta(\lambda_1-\lambda_3)=2\int_{\mathbb{S}^{n-1}}(\lambda_1-\lambda_3)m_1^2m_3^2\rho(m)\ud m,
\end{align*}
which implies $\beta=2\int_{\mathbb{S}^{n-1}}m_1^2m_3^2\rho(m)\ud m$ since $\lambda_1\neq\lambda_3$. Similarly, we have $\beta=2\int_{\mathbb{S}^{n-1}}m_2^2m_3^2\rho(m)\ud m$ as $\lambda_2\neq\lambda_3$. Thus
\begin{align}\label{eq-app1}
  \int_{\mathbb{S}^{n-1}}m_3^2(m_1^2-m_2^2)\rho(m)\ud m=0.
\end{align}
Interchanging the integration variables $m_1$ and $m_2$ in the above equality, we have
\begin{align}\label{eq-app2}
  \int_{\mathbb{S}^{n-1}}m_3^2(m_2^2-m_1^2)e^{\lambda_1m_2^2+\lambda_2m_1^2+\sum_{i\ge 3}\lambda_im_i^2}\ud m=0.
\end{align}
Adding up \eqref{eq-app1} and \eqref{eq-app2}, we get
\begin{align*} \int_{\mathbb{S}^{n-1}}m_3^2(m_1^2-m_2^2)(1-e^{(\lambda_2-\lambda_1)(m_1^2-m_2^2)})\rho(m)\ud m=0,
\end{align*}
which is impossible since $x(1-e^{(\lambda_2-\lambda_1)x})$ has the same sign with $\lambda_1-\lambda_2\neq 0$ for $x\neq 0$.
\end{proof}

\section*{Acknowledgments}
 W. Wang is supported by NSF of China under Grant No. 11931010 and 12271476.

\end{document}